\documentclass[12pt]{article}
\usepackage{amssymb,amsmath,amsthm,enumerate}

\newcommand{\X}{\mathbf{X}}
\newcommand{\Y}{\mathbf{Y}}

\newcommand{\Z}{\mathbb{Z}}
\newcommand{\cc}{\mathfrak{c}}

\newcommand{\A}{\mathfrak{A}}

\newcommand{\B}{\mathcal{B}}
\newcommand{\OO}{\mathcal{O}}

\newcommand{\Q}{\mathbb{Q}}
\newcommand{\F}{\mathbb{F}}
\renewcommand{\SS}{\mathbb{S}}
\newcommand{\M}{\mathcal{M}}

\renewcommand{\a}{\mathfrak{a}}
\renewcommand{\b}{\mathfrak{b}}

\newcommand{\x}{\vec{x}}

\renewcommand{\d}{\vec{d}}
\newcommand{\vbeta}{\vec{\beta}}
\newcommand{\vdelta}{\vec{\delta}}

\newcommand{\vomega}{\vec{\omega}}

\newcommand{\Dbar}{\overline{D}}
\newcommand{\Ebar}{\overline{E}}
\newcommand{\Sbar}{\overline{S}}

\newcommand{\ra}{\rightarrow}

\newcommand{\dst}{\displaystyle}

\DeclareMathOperator{\ch}{char}

\DeclareMathOperator{\Gal}{Gal}

\newtheorem{theorem}{Theorem}
\newtheorem{lemma}[theorem]{Lemma}
\newtheorem{prop}[theorem]{Proposition}
\newtheorem{cor}[theorem]{Corollary}

\theoremstyle{definition}
\newtheorem{remark}[theorem]{Remark}
\newtheorem{definition}[theorem]{Definition}

\numberwithin{equation}{section}
\numberwithin{theorem}{section}

\title{Galois scaffolds for cyclic $p^n$-extensions in
characteristic $p$}
\author{G. Griffith Elder \\
Department of Mathematics \\
University of Nebraska Omaha \\
Omaha, NE 68182 \\
USA \\[.2cm]
{\tt elder@unomaha.edu}
\and
Kevin Keating \\
Department of Mathematics \\
University of Florida \\
Gainesville, FL 32611 \\
USA \\[.2cm]
{\tt keating@ufl.edu}}

\begin{document}

\maketitle

\begin{abstract}
Let $K$ be a local field of characteristic $p$ and let
$L/K$ be a totally ramified Galois extension such that
$\Gal(L/K)\cong C_{p^n}$.  In this paper we find
sufficient conditions for $L/K$ to admit a Galois
scaffold, as defined in \cite{bce}.  This leads to
sufficient conditions for the ring of integers $\OO_L$
to be free of rank 1 over its associated order $\A_0$,
and to stricter conditions which imply that $\A_0$ is a
Hopf order in the group ring $K[C_{p^n}]$.
\end{abstract}

\section{Introduction}

Let $K$ be a field of characteristic $p$.  Witt
\cite{Witt} generalized Artin-Schreier theory by
proving that cyclic extensions $L/K$ of
degree $p^n$ can be described using what is now known as
the ring of Witt vectors of length $n$ over $K$, denoted
$W_n(K)$.  The elements of $W_n(K)$ are indeed vectors
with $n$ entries taken from $K$, with nonstandard
operations $\oplus,\otimes$ which make $W_n(K)$ a
commutative ring with 1.  Witt showed that for a cyclic
extension $L/K$ of degree $p^n$ there exists a vector
$\vbeta\in W_n(K)$ such that $L$ is generated over $K$
by the coordinates of any solution in $W_n(K^{sep})$ to
the equation $\phi(\x)=\x\oplus\vbeta$.  Here $K^{sep}$
is a separable closure of $K$ and
$\phi:W_n(K^{sep})\ra W_n(K^{sep})$ is the map induced
by the $p$-Frobenius on $K^{sep}$.  (See \cite{pdg} for
basic facts about Witt vectors.)

     Now suppose that $K$ is a local field of
characteristic $p$.  In \cite{bep2}, Byott and Elder
considered totally ramified Galois extensions $L/K$ of
degree $p^2$.  They gave sufficient conditions on
$\vbeta\in W_2(K)$ for the $C_{p^2}$-extension $L/K$
generated by the roots of $\phi(\x)=\x\oplus\vbeta$ to
admit a Galois scaffold.  In this paper we generalize
that result by giving sufficient conditions on
$\vbeta\in W_n(K)$ for the $C_{p^n}$-extension $L/K$
generated by the roots of $\phi(\x)=\x\oplus\vbeta$ to
admit a Galois scaffold.

     As explained in \cite{bce}, a scaffold enables one
to answer an array of questions all captured under the
heading {\em integral Galois module structure}.  By the
normal basis theorem, $L$ is free of rank 1 over
$K[C_{p^n}]$.  An integral version of the normal basis
theorem would state that the ring of integers $\OO_L$ is
free over some order $\B$ of $K[C_{p^n}]$.  Indeed, if
this holds then we must have $\B=\A_0$, where
\begin{equation} \label{A0}
\A_0=\{\gamma\in K[G]:\gamma(\OO_L)\subset\OO_L\}
\end{equation}
is the associated order of $\OO_L$ in $K[C_{p^n}]$
(see \S4 of \cite{mart}).  The problem of determining
the $\A_0$-module structure of $\OO_L$ is classical and
appears to be difficult in general.  However, when $L/K$
has a Galois scaffold with large enough precision (see
Definition~\ref{scaffold}), all that one might
reasonably expect is known.  Indeed, everything that can
be determined about integral Galois
module structure in a cyclic extension of degree $p$ can
be determined for an extension with a scaffold of
sufficiently large precision.  For instance, one can
give necessary and sufficient conditions in terms of the
ramification breaks of $L/K$ for $\OO_L$ to be free over
$\A_0$.  In this paper we give sufficient conditions for
$\OO_L$ to be free over $\A_0$ (Corollary~\ref{GMS}).
We omit the technicalities needed to formulate
necessary and sufficient conditions; see Section~3 of
\cite{bce} for details.  We also use scaffolds to
address a different classical problem: In
Corollary~\ref{Hopforder} we give sufficient conditions
for $\A_0$ to be a Hopf order in $K[G]$.

     Throughout the paper we let $K$ be a local field of
characteristic $p$ with separable closure $K^{sep}$.
For each finite subextension $F/K$ of $K^{sep}/K$ let
$v_F$ be the valuation on $K^{sep}$ normalized so that
$v_F(F^{\times})=\Z$.  Let $\OO_F$ denote the ring of
integers of $F$ and let $\M_F$ denote the maximal ideal
of $\OO_F$.

\section{Sufficient conditions for a Galois scaffold}
\label{scaff}

In this section we record the definition from \cite{bce}
of a Galois scaffold for a totally ramified Galois
extension $L/K$ of degree $p^n$.  We then describe the
sufficient conditions given in \cite{large} for $L/K$ to
admit a Galois scaffold.  In later sections we will use
Artin-Schreier-Witt theory to construct
$C_{p^n}$-extensions $L/K$ in characteristic $p$ which
satisfy the conditions from \cite{large}, and therefore
have Galois scaffolds.

     To give the definition of a Galois scaffold we use
notation from \S2 of \cite{bce}.  Let $L/K$ be a totally
ramified Galois extension of degree $p^n$ and set
$G=\Gal(L/K)$.  Let $b_1\le b_2\le\cdots\le b_n$ be the
lower ramification breaks of $L/K$, counted with
multiplicity. (See Chapter~IV of \cite{cl} for
information about ramification breaks of local field
extensions.)  Assume that $p\nmid b_i$ for
$1\le i\le n$.  Set $\SS_{p^n}=\{0,1,\ldots,p^n-1\}$ and
write $s\in\SS_{p^n}$ in base $p$ as
\[s=s_{(0)}p^0+s_{(1)}p^1+\cdots+s_{(n-1)}p^{n-1}\]
with $0\le s_{(i)}<p$.  Define $\b:\SS_{p^n}\ra\Z$ by
\[\b(s)=s_{(0)}p^0b_n+s_{(1)}p^1b_{n-1}+
\cdots+s_{(n-1)}p^{n-1}b_1.\]
Let $r:\Z\ra\SS_{p^n}$ be the function which maps
$t\in\Z$ onto its least nonnegative residue modulo
$p^n$.  The function $r\circ(-\b):\SS_{p^n}\ra\SS_{p^n}$
is a bijection since $p\nmid b_i$.  Therefore we may
define $\a:\SS_{p^n}\ra\SS_{p^n}$ to be the inverse of
$r\circ(-\b)$.  We extend $\a$ to a function from $\Z$
to $\SS_{p^n}$ by setting $\a(t)=\a(r(t))$ for $t\in\Z$.

     The following is a specialization of the general
definition of ``$A$-scaffold'' given in Definition~2.3
of \cite{bce}:

\begin{definition}[\cite{bce}, Definition 2.6]
\label{scaffold}
Let $\cc\ge1$.  A Galois scaffold
$(\{\Psi_i\},\{\lambda_t\})$ for $L/K$ with precision
$\cc$ consists of elements $\Psi_i\in K[G]$ for
$1\le i\le n$ and $\lambda_t\in L$ for all $t\in\Z$ such
that the following hold:
\begin{enumerate}[(i)]
\item $v_L(\lambda_t)=t$ for all $t\in\Z$.
\item $\lambda_{t_1}\lambda_{t_2}^{-1}\in K$
whenever $t_1\equiv t_2\pmod{p^n}$.
\item $\Psi_i(1)=0$ for $1\le i\le n$.
\item For each $1\le i\le n$ and $t\in\Z$ there exists
$u_{it}\in\OO_K^{\times}$ such that the following
congruence modulo $\lambda_{t+p^{n-i}b_i}\M_L^{\cc}$
holds:
\[\Psi_i(\lambda_t)\equiv\begin{cases} 
u_{it}\lambda_{t+p^{n-i}b_i}&
\mbox{if }\a(t)_{(n-i)}\ge1, \\ 
0&\mbox{if }\a(t)_{(n-i)}=0.
\end{cases} \]
\end{enumerate}
\end{definition}

     To prove that certain $C_{p^n}$-extensions admit
Galois scaffolds we will use a theorem from
\cite{large}.  In order to state this theorem we
introduce notation from \S2 of \cite{large}.  Let $L/K$
be a totally ramified $C_{p^n}$-extension whose lower
ramification breaks satisfy $b_i\equiv b_1\pmod{p^n}$
for $1\le i\le n$.  Let $1\le j\le n$ and let $K_j$
denote the fixed field of
$\langle\sigma^{p^j}\rangle$.  Let $\Y_j\in K_j$ satisfy
$v_{K_j}(\Y_j)=-b_j$.  Since $p\nmid b_j$ we have
$v_{K_j}((\sigma^{p^{j-1}}-1)(\Y_j))=b_j-b_j=0$.  Hence
there is $c_j\in\OO_K^{\times}$ such that $\X_j=c_j\Y_j$
satisfies
$(\sigma^{p^{j-1}}-1)(\X_j)\equiv1\pmod{\M_{K_j}}$.  For
$1\le j<i\le n$ we have $(\sigma^{p^{i-1}}-1)(\X_j)=0$.
We have $v_{K_j}(\X_j)=-b_j$, so for $1\le i\le j\le n$
we get $v_{K_j}((\sigma^{p^{i-1}}-1)(\X_j))=b_i-b_j$.
Since $p^n\mid b_i-b_j$ there are $\mu_{ij}\in K$ and
$\epsilon_{ij}\in K_j$ such that
\begin{equation} \label{sigpi1}
(\sigma^{p^{i-1}}-1)(\X_j)=\mu_{ij}+\epsilon_{ij}
\end{equation}
and $b_i-b_j=v_{K_j}(\mu_{ij})<v_{K_j}(\epsilon_{ij})$.
One views $\mu_{ij}$ as the ``main term'' and
$\epsilon_{ij}$ as the ``error term'' in our
representation of $(\sigma^{p^{i-1}}-1)(\X_j)$.  The
following theorem says that if the error terms are
sufficiently small compared to the main terms then $L/K$
admits a Galois scaffold.

\begin{theorem} \label{scafcond}
Let $\ch(K)=p$ and let $L/K$ be a totally ramified
$C_{p^n}$-extension whose lower ramification breaks
$b_1<b_2<\cdots<b_n$ satisfy $b_i\equiv b_1\pmod{p^n}$
for $1\le i\le n$.  Denote the upper ramification breaks
of $L/K$ by $u_1<u_2<\cdots<u_n$ and define
$\mu_{ij},\epsilon_{ij}$ as in (\ref{sigpi1}).  Suppose
there is $\cc\ge1$ such that for $1\le i\le j\le n$ we
have
\[v_L(\epsilon_{ij})-v_L(\mu_{ij})\ge
p^{n-1}u_i-p^{n-j}b_i+\cc.\]
Then $L/K$ has a Galois scaffold with precision $\cc$.
\end{theorem}

\begin{proof}
This follows by specializing Theorem~2.10 of
\cite{large} to our setting.  Note that the hypothesis
$p\nmid b_1$ from \cite{large} holds automatically since
$\ch(K)=p$.
\end{proof}

\section{A normal basis generator for $L/K$}
\label{normal}

In this section we study a certain class of
$C_{p^n}$-extensions.  For each extension $L/K$ in this
class we construct an element $\Y\in L$ that generates a
normal basis for $L/K$.  In the notation of
Section~\ref{scaff} we could write $L=K_n$, in which
case it would make sense to call our normal basis
generator $\Y_n$.  We have chosen not to do this in
order to keep the notation simple.

     Let $L/K$ be a finite totally ramified Galois
subextension of $K^{sep}/K$, with
$\Gal(L/K)\cong C_{p^n}$.  Let $W_n(K)$ denote the ring
of Witt vectors of length $n$ over $K$ and let
$\vbeta\in W_n(K)$ be a Witt vector which corresponds to
$L/K$ under Artin-Schreier-Witt theory.  For
$0\le i\le n-1$ let $\beta_i$ denote the $i$th
coordinate of $\vbeta$.  We may assume without loss of
generality that $\vbeta$ is reduced in the sense of
Proposition~4.1 from \cite{lt}.  This means that for each
$0\le i\le n-1$ we have either $v_K(\beta_i)\ge0$ or
$p\nmid v_K(\beta_i)$.  Define $\phi:K^{sep}\ra K^{sep}$
by $\phi(x)=x^p$.  Then $\phi$ induces a map from
$W_n(K)$ to itself by acting on coordinates.  Let
$\x\in W_n(K^{sep})$ satisfy $\phi(\x)=\x\oplus\vbeta$,
where $\oplus$ denotes Witt vector addition.  Then
$L=K(x_0,x_1,\ldots,x_{n-1})$ and there is a generator
$\sigma$ for $\Gal(L/K)\cong C_{p^n}$ such that
$\sigma(\x)=\x\oplus\vec1$, where $\vec1\in W_n(K)$ is
the multiplicative identity.

     Since $L/K$ is a totally ramified
$C_{p^n}$-extension we have $v_K(\beta_0)<0$.  Set
$\beta=\beta_0$ and assume there are
$\omega_i,\delta_i\in K$ such that
\begin{equation} \label{assume0}
\beta_i=\beta\omega_i^{p^{n-1}}+\delta_i,\;\;\;
v_K(\delta_i)>v_K(\beta_i)
\end{equation}
for $0\le i\le n-1$.  Note that $\omega_0=1$ and
$\delta_0=0$.  As in \cite{bep2,refp} we view
$\beta\omega_i^{p^{n-1}}$ as the ``main term'' of
$\beta_i$, and $\delta_i$ as the ``error term''.  Let
$\vomega\in K^n$ be the vector of $\omega_i$'s, let
$\vdelta\in K^n$ be the vector of $\delta_i$'s, and set
$\d=(\x\oplus\vbeta)-\x-\vbeta$.  We get
$\vbeta=\beta\phi^{n-1}(\vomega)+\vdelta$ and
\[\x=\begin{bmatrix}x_0\\x_1\\\vdots\\x_{n-1}
\end{bmatrix}\;\;\d=\begin{bmatrix}
0\\d_1\\\vdots\\d_{n-1}\end{bmatrix}\;\;
\vomega=\begin{bmatrix}1\\\omega_1\\\vdots\\\omega_{n-1}
\end{bmatrix}\;\;\vbeta=\begin{bmatrix}
\beta_0\\\beta_1\\\vdots\\\beta_{n-1}\end{bmatrix}
=\begin{bmatrix}\beta\\\beta\omega_1^{p^{n-1}}+\delta_1\\
\vdots\\\beta\omega_{n-1}^{p^{n-1}}+\delta_{n-1}\end{bmatrix}
\;\;\vdelta=\begin{bmatrix}0\\\delta_1\\
\vdots\\\delta_{n-1}\end{bmatrix}.\]
Since $\phi(\x)=\x+\beta\phi^{n-1}(\vomega)+\vdelta+\d$
we have $x_i^p-x_i=\omega_i^{p^n}\beta+\delta_i+d_i$ for
$0\le i\le n-1$.

     Let $b_1<b_2<\cdots<b_n$ and $u_1<u_2<\cdots<u_n$
be the sequences of lower and upper ramification breaks
of $L/K$.  These are related by the formulas $b_1=u_1$
and
\begin{equation} \label{bi1}
b_{i+1}-b_i=p^i(u_{i+1}-u_i)
\end{equation}
for $1\le i\le n-1$.  We assume throughout that the
following hold for $1\le i\le n-1$:
\begin{align} \label{assume1}
b_{i+1}&>p^nu_i \\
b_{i+1}&>-p^{n-1}v_K(\delta_i) \label{assume2}
\end{align}
It follows from the (well-known) Lemma~\ref{known} below
that $p^iu_{i+1}\ge b_{i+1}$.  Hence (\ref{assume1})
implies the weaker condition
\begin{equation} \label{assume3}
u_{i+1}>pu_i,
\end{equation}
which is sufficient for most of the steps of our
argument.  This last inequality is equivalent to the
statement that the sequence
$(p^{-i}u_i)_{0\leq i\leq n-1}$ is strictly increasing.

\begin{lemma} \label{known}
Let $1\le j\le n$.  Then $b_j\le p^{j-1}u_j$.
\end{lemma}

\begin{proof}
By (\ref{bi1}) we get
\begin{align*}
p^{j-1}u_j-b_j
&=p^{j-1}u_j-b_1-\sum_{h=1}^{j-1}\,(b_{h+1}-b_h) \\
&=p^{j-1}u_j-u_1-\sum_{h=1}^{j-1}p^h(u_{h+1}-u_h) \\
&=\sum_{h=1}^{j-1}\,(p^h-p^{h-1})u_h\ge0. \qedhere
\end{align*}
\end{proof}

     The following formula for the upper ramification
breaks of $L/K$ is an application of Corollary~5.1 of
\cite{lt}:
\[u_i=\max\{-p^{i-1}v_K(\beta_0),-p^{i-2}v_K(\beta_1),
\ldots,-p^1v_K(\beta_{i-2}),-p^0v_K(\beta_{i-1})\}.\]
It follows from the
assumption $v_K(\delta_{i-1})>v_K(\beta_{i-1})$ that
\begin{equation} \label{uibeta}
v_K(\beta\omega_{i-1}^{p^{n-1}})=v_K(\beta_{i-1})=-u_i
\end{equation}
for $1\le i\le n$.  Setting $v_K(\omega_i)=-m_i$ we get
$u_1=b_1$ and $u_i=b_1+p^{n-1}m_{i-1}$ for
$2\le i\le n$.  It follows that
\begin{equation} \label{mi}
-v_K(\omega_i)=m_i=p^{-n+1}(u_{i+1}-u_1)
\end{equation}
for $0\le i\le n-1$.  For $0\le i\le n$ set
$K_i=K(x_0,\ldots,x_{i-1})$.  Then $K_0=K$, $K_n=L$, and
$K_i$ is the fixed field of
$\langle\sigma^{p^i}\rangle$.

     For $0\le i\le n-1$ let
$S_i\in\Z[X_0,\ldots,X_i,Y_1,\ldots,Y_i]$ be the
$i$th Witt vector addition polynomial.  Then addition in
$W_n$ is given by
\[\begin{bmatrix}X_0\\X_1\\\vdots\\X_{n-1}\end{bmatrix}
\oplus\begin{bmatrix}Y_0\\Y_1\\\vdots\\Y_{n-1}\end{bmatrix}
=\begin{bmatrix}S_0\\S_1\\\vdots\\S_{n-1}\end{bmatrix}\]
and $S_i$ is defined in terms of $S_0,\ldots,S_{i-1}$ by
the recursion formula
\begin{equation} \label{recursion}
S_i=p^{-i}\left(\sum_{j=0}^ip^j
(X_j^{p^{i-j}}+Y_j^{p^{i-j}})-\sum_{j=0}^{i-1}
p^jS_j^{p^{i-j}}\right).
\end{equation}
Hence $S_i=X_i+Y_i+D_i$, with
$D_i\in\Z[X_0,\ldots,X_{i-1},Y_0,\ldots,Y_{i-1}]$.  In
particular, $S_0=X_0+Y_0$ and $D_0=0$.  We will use the
following elementary fact about $D_i$:

\begin{lemma} \label{Di}
Every monomial in $D_i$ has a factor $X_h$ for some
$0\le h\le i-1$, and a factor $Y_h$ for some
$0\le h\le i-1$.
\end{lemma}

\begin{proof}
Since $\vec0$ is the identity element for the operation
$\oplus$ on $W_n$, the only term of $S_i$ not divisible
by some $X_h$ is $Y_i$, and the only term of $S_i$ not
divisible by some $Y_h$ is $X_i$.
\end{proof}

     For the proof of the next lemma we let $\Sbar_i$,
$\Dbar_i$ denote the reductions modulo $p$ of $S_i$,
$D_i$.

\begin{lemma} \label{dival}
\begin{enumerate}[(a)]
\item $v_K(d_i)>-pu_i$ for $1\le i\le n-1$.
\item $v_K(x_i)=p^{-1}v_K(\beta_i)=-p^{-1}u_{i+1}$ for
$0\le i\le n-1$.
\end{enumerate}
\end{lemma}

\begin{proof}
We use induction on $i$.  Since $d_0=0$, the case $i=0$
is clear.  Let $1\le i\le n-1$ and assume the lemma
holds for $0\le j<i$.  In $W_n(K^{sep})$ we have
$\d=\x\oplus\vbeta-\x-\vbeta$, and hence
\[d_i=\Dbar_i(x_0,x_1,\ldots,x_{i-1},
\beta_0,\beta_1,\ldots,\beta_{i-1}).\]
Let $0\le j\le i-1$.  Then by
(\ref{uibeta}), (\ref{assume3}), and the inductive
hypotheses we get
\begin{align} \label{ineqbeta}
v_K(\beta_j)&=-u_{j+1}\ge-p^{j-i+1}u_i=p^{j-i}(-pu_i) \\
v_K(x_j)&=-p^{-1}u_{j+1}\ge-p^{j-i}u_i
>p^{j-i}(-pu_i). \label{ineqx}
\end{align}
If we assign $X_j$ and $Y_j$ the weight $p^j$ then an
inductive argument based on (\ref{recursion}) shows that
$\Sbar_i$ and $\Dbar_i$ are isobaric of weight $p^i$.
Using (\ref{ineqbeta}), (\ref{ineqx}), and
Lemma~\ref{Di} we deduce that each term of $d_i$ has
$K$-valuation greater than $-pu_i$.  It follows that
$v_K(d_i)>-pu_i$.  Using (\ref{assume3}) and
(\ref{uibeta}) we get $v_K(d_i)>-u_{i+1}=v_K(\beta_i)$,
and hence $v_L(d_i+\beta_i)=v_L(\beta_i)<0$.  Therefore
$v_K(x_i)=p^{-1}v_K(\beta_i)=-p^{-1}u_{i+1}$.
\end{proof}

     We now define
\begin{equation} \label{Y}
\Y=\det([\x,\vomega,\phi(\vomega),\ldots,
\phi^{n-2}(\vomega)]).
\end{equation}

\begin{lemma} \label{Yexp}
Let $t_i\in K$ be the $(i,0)$ cofactor of (\ref{Y}).
Then
\begin{enumerate}[(a)]
\item $\Y=t_0x_0+t_1x_1+\cdots+t_{n-1}x_{n-1}$.
\item $v_K(t_0)=-m_1-pm_2-\cdots-p^{n-2}m_{n-1}$.
\item For $0\leq i<j\leq n-1$ we have
$v_K(t_j)-v_K(t_i)=p^{-n}(b_{j+1}-b_{i+1})$.
\end{enumerate}
\end{lemma}

\begin{proof}
Since $x_i\in L$ and $\omega_i\in K$ for $0\le i\le n-1$
we have $\Y\in L$.  Keeping in mind that $\omega_0=1$ and
$0=m_0<m_1<\cdots<m_{n-1}$ we see that for
$0\le i\le n-1$ we have
\begin{align*}
v_K(t_i)&=v_K(\omega_0\omega_1^p\omega_2^{p^2}\ldots
\omega_{i-1}^{p^{i-1}}\omega_{i+1}^{p^i}\ldots
\omega_{n-1}^{p^{n-2}}) \\
&=-m_0-pm_1-p^2m_2-\cdots-p^{i-1}m_{i-1}-p^im_{i+1}-
\cdots-p^{n-2}m_{n-1}. 
\end{align*}
In particular, we have
\[v_K(t_0)=-m_1-pm_2-\cdots-p^{n-2}m_{n-1}.\]
Furthermore, by applying (\ref{mi}) and (\ref{bi1}) we
get
\begin{align*}
v_K(t_j)-v_K(t_{j-1})&=p^{j-1}(m_j-m_{j-1}) \\
&=p^{j-n}(u_{j+1}-u_j) \\
&=p^{-n}(b_{j+1}-b_j).
\end{align*}
It follows that for $0\le i<j\le n-1$ we have
$v_K(t_j)-v_K(t_i)=p^{-n}(b_{j+1}-b_{i+1})$.
\end{proof}

\begin{prop} \label{vLYprop}
Let $L/K$ be a $C_{p^n}$-extension which satisfies
assumptions (\ref{assume0}), (\ref{assume1}), and
(\ref{assume2}), and define $\Y$ as in (\ref{Y}).  Then
$L=K(\Y)$ and
\[v_L(\Y)=-b_1-p^nm_1-p^{n+1}m_2-\cdots-p^{2n-2}m_{n-1}.\]
\end{prop}

\begin{proof}
Since $x_i\in L$ and $\omega_i\in K$ for $0\le i\le n-1$
we have $\Y\in L$.  We claim that for $0\le i\le n-1$ we
have
\begin{equation} \label{phiiY}
\phi^i(\Y)=\det([\x+\d+\cdots+\phi^{i-1}(\d)+
\vdelta+\cdots+\phi^{i-1}(\vdelta),\phi^i(\vomega),
\ldots,\phi^{i+n-2}(\vomega)]).
\end{equation}
The case $i=0$ is given by (\ref{Y}).  Let
$0\le i\le n-2$ and assume that (\ref{phiiY}) holds for
$i$.  Then
\begin{align*}
\phi^{i+1}(\Y)
&=\phi(\det([\x+\d+\cdots+\phi^{i-1}(\d)+
\vdelta+\cdots+\phi^{i-1}(\vdelta),\phi^i(\vomega),
\ldots,\phi^{i+n-2}(\vomega)])) \\
&=\det([\phi(\x)+\phi(\d)+\cdots+\phi^i(\d)+
\phi(\vdelta)+\cdots+\phi^i(\vdelta),
\phi^{i+1}(\vomega),\ldots,\phi^{i+n-1}(\vomega)]) \\
&=\det([\x+\beta\phi^{n-1}(\vomega)+\vdelta+\d+\phi(\d)+
\cdots+\phi^i(\d)+\phi(\vdelta)+\cdots+\phi^i(\vdelta),
\\
&\hspace{9cm}\phi^{i+1}(\vomega),\ldots,\phi^{i+n-1}(\vomega)]).
\end{align*}
Since $i+1\le n-1\le i+n-1$ it follows that
\[\phi^{i+1}(\Y)
=\det([\x+\d+\phi(\d)+\cdots+\phi^i(\d)+
\vdelta+\phi(\vdelta)+\cdots+\phi^i(\vdelta),
\phi^{i+1}(\vomega),\ldots,\phi^{i+n-1}(\vomega)]).\]
Hence (\ref{phiiY}) holds with $i$ replaced by $i+1$.
It follows by induction that (\ref{phiiY}) holds for
$i=n-1$.  Therefore we have
\begin{align}
\phi^n(\Y)&=\phi(\det([\x+\d+\cdots+\phi^{n-2}(\d)+
\vdelta+\cdots+\phi^{n-2}(\vdelta),
\phi^{n-1}(\vomega),\ldots,\phi^{2n-3}(\vomega)]))
\nonumber \\
&=\det([\phi(\x)+\phi(\d)+\cdots+\phi^{n-1}(\d)+
\phi(\vdelta)+\cdots+\phi^{n-1}(\vdelta),
\phi^n(\vomega),\ldots,\phi^{2n-2}(\vomega)]) \nonumber \\
&=\det([\x+\beta\phi^{n-1}(\vomega)+\d+\phi(\d)+
\cdots+\phi^{n-1}(\d)+\vdelta+\phi(\vdelta)+\cdots
+\phi^{n-1}(\vdelta), \nonumber \\
&\hspace{9cm}\phi^n(\vomega),\ldots,\phi^{2n-2}(\vomega)]).
\label{phinY}
\end{align}

     Observe that the $(i,0)$ cofactor of (\ref{phinY})
is $t_i^{p^n}$, where $t_i$ is the $(i,0)$ cofactor of
(\ref{Y}), as in Lemma~\ref{Yexp}.  Therefore
\begin{equation} \label{Ypn}
\Y^{p^n}=t_0^{p^n}(x_0+\beta)+\sum_{i=1}^{n-1}
t_i^{p^n}\left(x_i+\beta\omega_i^{p^{n-1}}
+\sum_{j=0}^{n-1}(d_i^{p^j}+\delta_i^{p^j})\right).
\end{equation}
Using Lemma~\ref{Yexp}(b) we get
\begin{align*}
v_K(t_0^{p^n}\beta)&=v_K(\beta)+p^nv_K(t_0) \\
&=-b_1-p^nm_1-p^{n+1}m_2-\cdots-p^{2n-2}m_{n-1}.
\end{align*}
To prove the proposition it suffices to show that the
other terms in (\ref{Ypn}) all have $K$-valuation
greater than $v_K(t_0^{p^n}\beta)$.

     We begin by showing that
$v_K(t_i^{p^n}\cdot\beta\omega_i^{p^{n-1}})
>v_K(t_0^{p^n}\beta)$ for $1\le i\le n-1$.  In fact, by
Lemma~\ref{Yexp}(c) and (\ref{mi}) we have
\begin{align*}
v_K(t_i^{p^n}\omega_i^{p^{n-1}})-v_K(t_0^{p^n})
&=p^n(v_K(t_i)-v_K(t_0))+p^{n-1}v_K(\omega_i) \\
&=(b_{i+1}-b_1)-(u_{i+1}-u_1) \\
&=b_{i+1}-u_{i+1}>0.
\end{align*}
Hence $v_K(t_i^{p^n}\cdot\beta\omega_i^{p^{n-1}})
>v_K(t_0^{p^n}\beta)$.

     Let $0\le i\le n-1$.  By Lemma~\ref{dival}(b) we
have $v_K(x_i)=p^{-1}v_K(\beta_i)>v_K(\beta_i)$.  It
follows from the preceding paragraph that
\[v_K(t_i^{p^n}x_i)>v_K(t_i^{p^n}\beta_i)
=v_K(t_i^{p^n}\cdot\beta\omega_i^{p^{n-1}})>
v_K(t_0^{p^n}\beta).\]
By Lemma~\ref{Yexp}(c) and assumption (\ref{assume2}) we
get
\begin{align*}
v_K(t_i^{p^n}\delta_i^{p^j})-v_K(t_0^{p^n}\beta)
&=p^n(v_K(t_i)-v_K(t_0))+p^jv_K(\delta_i)-v_K(\beta) \\
&=b_{i+1}-b_1+p^jv_K(\delta_i)-(-b_1) \\
&=b_{i+1}+p^jv_K(\delta_i)>0
\end{align*}
for $0\le i,j\le n-1$.  Hence
$v_K(t_i^{p^n}\delta_i^{p^j})>v_K(t_0^{p^n}\beta)$.

     It remains to show that $v_K(t_i^{p^n}d_i^{p^j})>
v_K(t_0^{p^n}\beta)$ for $1\le i,j\le n-1$.  By
Lemma~\ref{Yexp}(c), Lemma~\ref{dival}(a), and
assumption (\ref{assume1}) we have
\begin{align*}
v_K(t_i^{p^n}d_i^{p^j})-v_K(t_0^{p^n}\beta)
&=p^n(v_K(t_i)-v_K(t_0))+p^jv_K(d_i)-v_K(\beta) \\
&>(b_{i+1}-b_1)-p^{j+1}u_i+b_1 \\
&=b_{i+1}-p^{j+1}u_i>0.
\end{align*}
Hence $v_K(t_i^{p^n}d_i^{p^j})>v_K(t_0^{p^n}\beta)$.  It
follows that
\begin{equation} \label{vLY}
v_L(\Y)=p^nv_K(\Y)=v_K(\Y^{p^n})=v_K(t_0^{p^n}\beta).
\end{equation}
Since $p\nmid v_L(\Y)$ we get $L=K(\Y)$.
\end{proof}

     For later use we record the following variant of
Proposition~\ref{vLYprop}:

\begin{cor} \label{vLYtn}
$v_L(\Y)=v_L(t_{n-1})-b_n$
\end{cor}

\begin{proof}
By (\ref{vLY}) and Lemma~\ref{Yexp}(c) we have
\begin{align*}
v_L(\Y)-v_L(t_{n-1})
&=p^nv_K(t_0)+v_K(\beta)-v_L(t_{n-1}) \\
&=v_K(\beta)+p^nv_K(t_0)-p^nv_K(t_{n-1}) \\
&=-b_1-(b_n-b_1)=-b_n. \qedhere
\end{align*}
\end{proof}

     It follows from the proposition that $\Y$ is a
``valuation criterion'' element of $L$, and hence
generates a normal basis for $L/K$.

\begin{cor}
$\{\sigma(\Y):\sigma\in\Gal(L/K)\}$ is a $K$-basis for
$L$.
\end{cor}

\begin{proof}
Since $t_{n-1}\in K$ the previous corollary implies that
$v_L(\Y)\equiv-b_n\pmod{p^n}$.  Hence the claim follows
from Theorem~2 of \cite{vc}.
\end{proof}

\section{The Galois action on $\Y$}

The goal of this section is to approximate
$(\sigma^{p^i}-1)(\Y)$ for $0\le i\le n-1$.  In the next
section we will use these approximations together with
Theorem~\ref{scafcond} to get a Galois scaffold for
$L/K$.

     Since $(\sigma^{p^i}-1)(x_j)=0$ for
$0\le j\le i-1$, it follows from Lemma~\ref{Yexp}(a)
that
\begin{align} \nonumber
(\sigma^{p^i}-1)(\Y)
&=(\sigma^{p^i}-1)(t_0x_0+t_1x_1+\cdots+t_{n-1}x_{n-1}) \\
&=t_i(\sigma^{p^i}-1)(x_i)+\cdots
+t_{n-1}(\sigma^{p^i}-1)(x_{n-1}). \label{sigpiY}
\end{align}
Therefore to approximate $(\sigma^{p^i}-1)(\Y)$ it
suffices to approximate $(\sigma^{p^i}-1)(x_j)$ for
$i\le j\le n-1$.  To do this we will use the following
two facts about Witt vectors.

\begin{lemma} \label{witta}
For $0\le i\le j$ let $S_j$ be the $j$th Witt addition
polynomial over $\Z$.  Then the coefficient of
$X_i^{p-1}X_{i+1}^{p-1}\ldots X_{j-1}^{p-1}Y_i$ in $S_j$
is $(-1)^{j-i}$.
\end{lemma}

\begin{proof}
We fix $i$ and use induction on $j$.  For $j=i$ the
coefficient of $Y_i$ in $S_i$ is $1=(-1)^{i-i}$.  Let
$j\ge i+1$ and assume that the claim holds for $j-1$.
Since $S_h$ does not depend on $X_{j-1}$ for
$0\le h\le j-2$, the only summand in the recursion
formula
\begin{equation} \label{recursion2}
S_j=p^{-j}\left(\sum_{h=0}^jp^h
(X_h^{p^{j-h}}+Y_h^{p^{j-h}})-\sum_{h=0}^{j-1}
p^hS_h^{p^{j-h}}\right)
\end{equation}
that can include the term
$X_i^{p-1}X_{i+1}^{p-1}\ldots X_{j-1}^{p-1}Y_i$ is
$-p^{-1}S_{j-1}^p$.  We have $S_{j-1}=X_{j-1}+\gamma$,
where $\gamma$ does not depend on $X_{j-1}$.  Hence
\[S_{j-1}^p=\sum_{h=0}^p\binom{p}{h}X_{j-1}^h\gamma^{p-h},\]
and the only summand on the right that can include the
term $X_i^{p-1}X_{i+1}^{p-1}\ldots X_{j-1}^{p-1}Y_i$ is
$\dst\binom{p}{p-1}X_{j-1}^{p-1}\gamma^1$.
By the inductive assumption, the coefficient of
$X_i^{p-1}X_{i+1}^{p-1}\ldots X_{j-2}^{p-1}Y_i$ in $\gamma$
is $(-1)^{j-1-i}$.  Hence the coefficient of
$X_i^{p-1}X_{i+1}^{p-1}\ldots X_{j-1}^{p-1}Y_i$ in $S_j$ is
\[-\frac1p\binom{p}{p-1}(-1)^{j-1-i}
=(-1)^{j-i}. \qedhere\]
\end{proof}

\begin{lemma} \label{wittb}
Let $E_{ij}$ be the polynomial
obtained from $D_j=S_j-X_j-Y_j$ by setting $Y_h=0$ for
$0\le h\le i-1$.  Then
\[E_{ij}\in \Z[X_i,X_{i+1},\ldots,X_{j-1},
Y_i,Y_{i+1},\ldots,Y_{j-1}].\]
\end{lemma}

\begin{proof}
Let $T_{ij}=X_j+Y_j+E_{ij}$ be the polynomial obtained
from $S_j$ by setting $Y_h=0$ for $0\le h\le i-1$.  By
Lemma~\ref{Di}, $T_{ih}=X_h$ for $0\le h\le i-1$.  It
follows from (\ref{recursion2}) that for $j\ge i$ we
have
\[T_{ij}=p^{-j}\left(\sum_{h=i}^jp^h
(X_h^{p^{j-h}}+Y_h^{p^{j-h}})-\sum_{h=i}^{j-1}
p^hT_{ih}^{p^{j-h}}\right).\]
In particular, $T_{ii}=X_i+Y_i$.  Using
induction on $j$ we get $T_{ij}\in
\Q[X_i,\ldots,X_j,Y_i,\ldots,Y_j]$ for $j\ge i$.  Since
$D_j\in\Z[X_0,\ldots,X_{j-1},Y_0,\ldots,Y_{j-1}]$ the
lemma follows from this.
\end{proof}

\begin{prop} \label{sigpi}
Let $L/K$ be a $C_{p^n}$-extension which satisfies
assumptions (\ref{assume0}), (\ref{assume1}), and (\ref{assume2}).  Let
$\sigma$ be a generator for $\Gal(L/K)$ such that
$\sigma(\x)=\x\oplus\vec1$, where $\vec1\in W_n(K)$ is
the multiplicative identity.  Then the following hold:
\begin{enumerate}[(a)]
\item For $0\le i\le n-1$ we have
$(\sigma^{p^i}-1)(x_i)=1$.
\item For $0\le i<j\le n-1$ we have
\[v_K((\sigma^{p^i}-1)(x_j))
=-(1-p^{-1})(u_{i+1}+\cdots+u_j).\]
\end{enumerate}
\end{prop}

\begin{proof}
(a) It follows from the assumption on $\sigma$ that
$\sigma^{p^i}(\x)=\x\oplus\vec{p^i}$, where
$\vec{p^i}=p^i\cdot\vec1$ is the element of
$W_n(K)$ which has a 1 in position $i$ and 0 in
all other positions.  Hence by Lemma~\ref{Di} we get
\[\sigma^{p^i}(x_i)
=x_i+1+\Dbar_i(x_0,\ldots,x_{i-1},0,\ldots,0)=x_i+1.\]
(b) Let $\tau_j$ denote the $j$th entry of
$\x\oplus\vec{p^i}$.  It follows from
Lemma~\ref{wittb} that $\tau_j-x_j$ can be expressed
as a polynomial in $x_i,\ldots,x_{j-1}$ with
coefficients in $\F_p$.  In fact, letting $\Ebar_{ij}$
be the image of $E_{ij}$ in
$\F_p[X_i,\ldots,X_j,Y_i,\ldots,Y_j]$ we get
\begin{equation} \label{tauj}
\tau_j-x_j
=\Ebar_{ij}(x_i,\ldots,x_{j-1},1,0,\ldots,0).
\end{equation}
As in the proof of Lemma~\ref{dival}, for $0\le h\le j$
we assign $X_h$ and $Y_h$ the weight $p^h$.  This makes
the $j$th Witt addition polynomial $\Sbar_j$ isobaric of
weight $p^j$.  Hence $\Ebar_{ij}$ is  also isobaric of
weight $p^j$.  It follows from Lemma~\ref{Di} and
Lemma~\ref{wittb} that every term in $\Ebar_{ij}$ has a
factor $Y_h$ for some $i\le h\le j-1$.  Thus if we
assign the weight $p^h$ to $x_h$ in (\ref{tauj}), every
term in $\tau_j-x_j$ has weight $<p^j$.

     We wish to find a lower bound for the valuations of
terms occurring in $\tau_j-x_j$.  If
$x_i^{a_i}x_{i+1}^{a_{i+1}}\ldots x_{j-1}^{a_{j-1}}$ is
such a term then $a_i,a_{i+1},\ldots,a_{j-1}$ are
nonnegative integers satisfying
\begin{equation} \label{bound}
p^ia_i+p^{i+1}a_{i+1}+\cdots+p^{j-1}a_{j-1}<p^j.
\end{equation}
Assume that our choice of $a_h$ for $i\le h\le j-1$
minimizes
\begin{equation} \label{min}
v_K(x_i^{a_i}x_{i+1}^{a_{i+1}}\ldots
x_{j-1}^{a_{j-1}})=-p^{-1}(a_iu_{i+1}+a_{i+1}u_{i+2}+\cdots+
a_{j-1}u_j)
\end{equation}
subject to the constraint (\ref{bound}).  Suppose
$a_h\ge p$ for some $i\le h\le j-1$; then $h<j-1$ by
(\ref{bound}).  Set $a_h'=a_h-p$, $a_{h+1}'=a_{h+1}+1$,
and $a_t'=a_t$ for $i\le t\le j-1$, $t\not\in\{h,h+1\}$.
Then $a_i',a_{i+1}',\ldots,a_{j-1}'$ are nonnegative
integers such that
\[p^ia_i'+p^{i+1}a_{i+1}'+\cdots+p^{j-1}a_{j-1}'
=p^ia_i+p^{i+1}a_{i+1}+\cdots+p^{j-1}a_{j-1}<p^j.\]
Since $h<j<n$ we have $h+1\le n-1$.  Hence by
Lemma~\ref{dival}(b) and (\ref{assume3}) we get
\[v_K(x_{h+1})=-p^{-1}u_{h+2}<-u_{h+1}=pv_K(x_h).\]
Therefore
\[v_K(x_i^{a_i'}x_{i+1}^{a_{i+1}'}\ldots
x_{j-1}^{a_{j-1}'})<v_K(x_i^{a_i}x_{i+1}^{a_{i+1}}
\ldots x_{j-1}^{a_{j-1}}).\]
This contradicts the minimality of
$v_K(x_i^{a_i}x_{i+1}^{a_{i+1}}\ldots x_{j-1}^{a_{j-1}})$,
so we must have $a_h\le p-1$ for $i\le h\le j-1$.  On
the other hand, letting $a_h=p-1$ for $i\le h\le j-1$
satisfies (\ref{bound}), so the minimum is achieved in
(\ref{min}) with this choice.  Furthermore, this is the
unique choice of nonnegative values for $a_h$ satisfying
(\ref{bound}) which minimizes (\ref{min}).  By
Lemma~\ref{witta} the coefficient of
$x_i^{p-1}x_{i+1}^{p-1}\ldots x_{j-1}^{p-1}$ in the
formula (\ref{tauj}) for $\tau_j-x_j$ is $(-1)^{j-i}$.
Hence by Lemma~\ref{dival}(b) we get
\begin{align*}
v_K((\sigma^{p^i}-1)(x_j))
&=v_K(x_i^{p-1}x_{i+1}^{p-1}\ldots x_{j-1}^{p-1}) \\
&=-(1-p^{-1})(u_{i+1}+\cdots+u_j).\qedhere
\end{align*}
\end{proof}

\begin{cor} \label{tn1}
$(\sigma^{p^{n-1}}-1)(\Y)=t_{n-1}\in K^{\times}$.
\end{cor}

\begin{proof}
Using (\ref{sigpiY}) and Proposition~\ref{sigpi}(a) we
get
\[(\sigma^{p^{n-1}}-1)(\Y)
=t_{n-1}(\sigma^{p^{n-1}}-1)(x_{n-1})=t_{n-1}.\]
Since $L=K(\Y)$ we have $t_{n-1}\not=0$.
\end{proof}

\begin{prop} \label{shifts}
Let $L/K$ be a $C_{p^n}$-extension which satisfies
assumptions (\ref{assume0}), (\ref{assume1}), and (\ref{assume2}).  Then
for $1\le i\le n-1$ we have
\[(\sigma^{p^{i-1}}-1)(\Y)\equiv t_{i-1}
\pmod{t_{i-1}\M_L^{p^i(u_{i+1}-u_i)-p^nu_i+p^{n-1}u_i}}.\]
\end{prop}

\begin{proof}
It follows from (\ref{sigpiY}) and
Proposition~\ref{sigpi}(a) that
\begin{equation} \label{vKdiff}
v_K(\sigma^{p^{i-1}}(\Y)-\Y-t_{i-1})\ge
\min\{v_K(t_j(\sigma^{p^{i-1}}-1)(x_j)):i\le j\le n-1\}.
\end{equation}
Using Proposition~\ref{sigpi}(b), Lemma~\ref{Yexp}(c),
and assumption (\ref{assume1}) we get
\begin{align*}
v_K(t_j(\sigma^{p^{i-1}}-1)(x_j))-
v_K(t_{j-1}(\sigma^{p^{i-1}}-1)(x_{j-1}))
&=v_K(t_j)-v_K(t_{j-1})-(1-p^{-1})u_j \\
&=p^{-n}(b_{j+1}-b_j)-(1-p^{-1})u_j \\
&>u_j-p^{-n}b_j-(1-p^{-1})u_j \\
&=-p^{-n}b_j+p^{-1}u_j \\
&=p^{-n}(p^{n-1}u_j-b_j)
\end{align*}
for $i+1\le j\le n-1$.  This last quantity is positive
by Lemma~\ref{known}.  Hence by (\ref{vKdiff}),
Proposition~\ref{sigpi}(b), Lemma~\ref{Yexp}(c), and
(\ref{bi1}) we have
\begin{align*}
v_K(\sigma^{p^{i-1}}(\Y)-\Y-t_{i-1})
&\ge v_K(t_i(\sigma^{p^{i-1}}-1)(x_i)) \\
&=v_K(t_i)-(1-p^{-1})u_i \\
&=v_K(t_{i-1})+p^{-n}(b_{i+1}-b_i)-(1-p^{-1})u_i \\
&=v_K(t_{i-1})+p^{i-n}(u_{i+1}-u_i)-(1-p^{-1})u_i.
\end{align*}
It follows that $\sigma^{p^{i-1}}(\Y)-\Y-t_{i-1}\in
t_{i-1}\M_L^{p^i(u_{i+1}-u_i)-p^nu_i+p^{n-1}u_i}$.
\end{proof}

\begin{cor} \label{X}
Let $L/K$ be a $C_{p^n}$-extension which satisfies
assumptions (\ref{assume0}), (\ref{assume1}), and (\ref{assume2}).  Let
$\X=t_{n-1}^{-1}\Y$, and for $1\le i\le n$ set
$\mu_i=t_{n-1}^{-1}t_{i-1}$ and
$\epsilon_i=(\sigma^{p^{i-1}}-1)(\X)-\mu_i$.  Then
$\epsilon_n=0$, and for $1\le i\le n-1$ we have
\[v_L(\epsilon_i)-v_L(\mu_i)\ge
p^i(u_{i+1}-u_i)-(p^n-p^{n-1})u_i.\]
\end{cor}

\begin{proof}
The first claim follows from Corollary~\ref{tn1}, and
the second follows from Proposition~\ref{shifts}.
\end{proof}

\section{Main results}

Recall that for $0\le j\le n$, $K_j$ is the fixed field
of the subgroup $\langle\sigma^{p^j}\rangle$ of
$\Gal(L/K)=\langle\sigma\rangle$.
Since $\Gal(L/K)$ is a cyclic $p$-group,
$\langle\sigma^{p^j}\rangle$ is necessarily a
ramification subgroup of $\Gal(L/K)$.  Thus the upper
ramification breaks of $K_j/K$ are $u_1,u_2,\ldots,u_j$,
and the lower ramification breaks are
$b_1,b_2,\ldots,b_j$ (see \S4 in Chapter IV of
\cite{cl}).  We can describe the extension
$K_j/K$ by truncating the Witt vector equations in
Section~\ref{normal}.  Thus $K_j=K(x_0,\ldots,x_{j-1})$
with $\phi(\x)=\x\oplus\vbeta$,
$\d=(\x\oplus\vbeta)-\x-\vbeta$, and
\[\x=\begin{bmatrix}x_0\\x_1\\\vdots\\x_{j-1}
\end{bmatrix}\;\;\d=\begin{bmatrix}
0\\d_1\\\vdots\\d_{j-1}\end{bmatrix}\;\;
\vomega=\begin{bmatrix}1\\\omega_1\\\vdots\\\omega_{j-1}
\end{bmatrix}\;\;\vbeta=\begin{bmatrix}
\beta_0\\\beta_1\\\vdots\\\beta_{j-1}\end{bmatrix}.\]
Assumptions (\ref{assume0}), (\ref{assume1}), and (\ref{assume2}) continue
to be valid for $K_j/K$.  As a result, we may use the
methods of Section~\ref{normal} to construct a
generator $\Y_j$ for the extension $K_j/K$, as in
(\ref{Y}).  In doing so, we add a subscript $j$.  Let
$t_{i,j}$ denote the $(i,0)$ cofactor of the matrix that
defines $\Y_j$.  (Thus $t_i$ from Section~\ref{normal}
will now be expressed as $t_{i,n}$.)  By
Corollary~\ref{vLYtn}
we have $v_{K_j}(\Y_j)=v_{K_j}(t_{j-1,j})-b_j$.  Corollary~\ref{tn1} yields
\[(\sigma^{p^{j-1}}-1)(\Y_j)=t_{j-1,j}\in K^\times.\]
Thus $\X_j=t_{j-1,j}^{-1}\Y_j$ is defined,
$v_{K_j}(\X_j)=-b_j$, $(\sigma^{p^j-1}-1)(\X_j)=1$, and
$K_j=K(\Y_j)=K(\X_j)$.

     We are now prepared to state and prove our main result.

\begin{theorem} \label{main}
Let $\ch(K)=p$ and let $L/K$ be a totally ramified
$C_{p^n}$-extension.  Let $\vbeta\in W_n(K)$ be a
reduced Witt vector which corresponds to $L/K$ and let
$\beta_0,\beta_1,\ldots,\beta_{n-1}$ be the coordinates
of $\vbeta$.  Set $\beta=\beta_0$ and assume there are
$\omega_i,\delta_i\in K$ such that
$\beta_i=\beta\omega_i^{p^{n-1}}+\delta_i$ and
$v_K(\delta_i)>v_K(\beta_i)$ for $1\le i\le n-1$.
Assume further that assumptions (\ref{assume1}) and
(\ref{assume2}) hold.  Then there is a Galois scaffold
$(\{\lambda_w\},\{\Psi_i\})$ for $L/K$ with precision
\[\cc=\min\{b_{i+1}-p^nu_i:1\le i\le n-1\}\ge1,\]
where $b_1<b_2<\cdots<b_n$ and $u_1<u_2<\cdots<u_n$ are
the upper and lower ramification breaks of $L/K$.
\end{theorem}

\begin{proof}
Let $1\le i\le j\le n$. The congruence hypothesis
$b_i\equiv b_j\pmod{p^n}$ in Theorem~\ref{scafcond} is
satisfied as a consequence of Lemma~\ref{Yexp}(c).
Corollary~\ref{X} yields
\[(\sigma^{p^{i-1}}-1)(\X_j)=\mu_{ij}+\epsilon_{ij}\]
for $1\leq i\leq j\leq n$, where
$\mu_{ij}=t_{j-1,j}^{-1}t_{i-1,j}$, $\mu_{jj}=1$,
$\epsilon_{jj}=0$, and 
\[v_{K_j}(\epsilon_{ij})-v_{K_j}(\mu_{ij})\geq p^i(u_{i+1}-u_i)-(p^j-p
^{j-1})u_i.\]
Using (\ref{bi1}) we get
\begin{align*}
v_L(\epsilon_{ij})-v_L(\mu_{ij})
&\ge p^{n-j+i}(u_{i+1}-u_i)-p^nu_i+p^{n-1}u_i \\
&=p^{n-j}(b_{i+1}-b_i)-p^nu_i+p^{n-1}u_i \\
&=p^{n-1}u_i-p^{n-j}b_i+p^{n-j}(b_{i+1}-p^ju_i)
\end{align*}
for $1\le i<j\le n$.  Therefore by
Theorem~\ref{scafcond} the extension $L/K$ has a
scaffold of precision $\cc$, with
\begin{align*}
\cc&=\min\{p^{n-j}(b_{i+1}-p^ju_i):1\le i<j\le n\} \\
&=\min\{b_{i+1}-p^nu_i:1\le i<n\}.
\end{align*}
Finally, we have $\cc\ge1$ by (\ref{assume1}).
\end{proof}

\begin{remark}
Theorem~\ref{main} for $n=2$ is in complete agreement
with Theorem~2.1 in \cite{bep2}. First, the hypotheses
are the same: Assumptions (\ref{assume1}),
(\ref{assume2}), which are required here for
Theorem~\ref{main}, reduce to (7), (8) in \cite{bep2},
which are required there for Theorem~2.1.  However
because the definition of a scaffold and the notion of a
scaffold's precision had not been fully formulated when
\cite{bep2} was written, a comparison of the
resulting scaffolds, including their precisions, is not
so immediate.  One has to interpret the content of
Theorem~2.1 in \cite{bep2} appropriately.  There one
sees that $\Psi_2$ increases valuations by $b_2$, while
$\Psi_1$ increases valuations by $pb_1$. As a result,
one would expect $\Psi_1^p$ to increase valuations by
$p^2b_1$, but since $\Psi_1^p=\Psi_2$, it actually
increases valuations by more, namely $b_2$. This
difference $\cc=b_2-p^2b_1$ is the precision of the
Galois scaffold given in \cite{bep2}, and it is the same
as the precision given in Theorem~\ref{main}.  (Beware
that both Remark~3.5 and Appendix~A.2.3 in \cite{large}
erroneously state that $\cc=b_2-pb_1$ is the precision
of the scaffold in \cite{bep2}.)
\end{remark}

     The scaffolds provided by Theorem~\ref{main} can be
used to get information about Galois module structure.
Recall that the associated order $\A_0$ of $\OO_L$ in
$K[C_{p^n}]$ is defined in (\ref{A0}).

\begin{cor}\label{GMS}
Let $L/K$ be a $C_{p^n}$-extension which satisfies the
hypotheses of Theorem~\ref{main}.  Let $r(u_1)$ denote
the least nonnegative residue modulo $p^n$ of the upper
ramification break $u_1$.  Strengthen assumption
(\ref{assume1}) by requiring that
$b_{i+1}-p^nu_i\ge r(u_1)$ for $1\le i\le n-1$.  Assume
further that $r(u_1)\mid p^m-1$ for some $1\le m\le n$.
Then $\OO_L$ is free over its associated order $\A_0$.
\end{cor}

\begin{proof}
Since $b_n\equiv b_1\pmod{p^n}$ and $b_1=u_1$ we have
$r(b_n)=r(u_1)$.  Theorem~\ref{main} gives us a
scaffold with precision $\cc\ge r(b_n)$,
so the corollary follows from Theorem~4.8 of \cite{bce}.
\end{proof}

     Let $H$ be an $\OO_K$-order in $K[C_{p^n}]$.  Say
that $H$ is a {\em Hopf order} if $H$ is a Hopf algebra
over $\OO_K$ with respect to the operations inherited
from the $K$-Hopf algebra $K[C_{p^n}]$.  Say that the
Hopf order $H\subset K[C_{p^n}]$ is {\em realizable} if
there is a $C_{p^n}$-extension $L/K$ such that $H$ is
isomorphic to the associated order $\A_0$ of $\OO_L$ in
$K[C_{p^n}]$.  The scaffolds from Theorem~\ref{main} can
be used to construct realizable Hopf orders in
$K[C_{p^n}]$:

\begin{cor}\label{Hopforder}
Let $L/K$ be a $C_{p^n}$-extension which satisfies the
hypotheses of Corollary~\ref{GMS}.  Assume further
that $u_1\equiv-1\pmod{p^n}$.  Then the associated order
$\A_0$ of $\OO_L$ in $K[C_{p^n}]$ is a Hopf order.
\end{cor}

\begin{proof}
It follows from the preceding corollary that $\OO_L$ is
free over $\A_0$.  Since
$b_i\equiv b_1\equiv-1\pmod{p^n}$ for $1\le i\le n$ the
different ideal of $L/K$ is generated by an element of
$K$.  Hence by Theorem~A and
Proposition~3.4.1 of \cite{bon1} we deduce that $\A_0$
is a Hopf order in $K[C_{p^n}]$.
\end{proof}

\end{document}